\documentclass[12pt,twoside]{amsart}
\usepackage{amssymb,amscd,amsxtra,calc,mathrsfs}
\usepackage{amsmath}
\usepackage{xypic}

\title[Bounds for the volumes]{Optimal bounds for the volumes of 
K\"ahler-Einstein Fano manifolds}
\author{Kento Fujita} 
\date{\today}
\address{Department of Mathematics, Faculty of Science, 
Kyoto University, Kyoto 606-8502, Japan}
\email{fujita@math.kyoto-u.ac.jp}

\newcommand{\pr}{\mathbb{P}}

\newcommand{\Z}{\mathbb{Z}}
\newcommand{\Q}{\mathbb{Q}}
\newcommand{\R}{\mathbb{R}}
\newcommand{\C}{\mathbb{C}}

\newcommand{\A}{\mathbb{A}}
\newcommand{\G}{\mathbb{G}}
\newcommand{\E}{\mathbb{E}}

\newcommand{\ND}{\operatorname{N}^1}

\newcommand{\Supp}{\operatorname{Supp}}
\newcommand{\Exc}{\operatorname{Exc}}

\newcommand{\Proj}{\operatorname{Proj}}

\newcommand{\lct}{\operatorname{lct}}
\newcommand{\DF}{\operatorname{DF}}
\newcommand{\Ding}{\operatorname{Ding}}
\newcommand{\ord}{\operatorname{ord}}

\newcommand{\vol}{\operatorname{vol}}
\newcommand{\Image}{\operatorname{Image}}
\newcommand{\SM}{\operatorname{sm}}
\newcommand{\sI}{\mathcal{I}}
\newcommand{\sJ}{\mathcal{J}}

\newcommand{\sO}{\mathcal{O}}

\newcommand{\sX}{\mathcal{X}}
\newcommand{\sY}{\mathcal{Y}}
\newcommand{\sL}{\mathcal{L}}

\newcommand{\sM}{\mathcal{M}}
\newcommand{\sF}{\mathcal{F}}
\newcommand{\sG}{\mathcal{G}}

\newcommand{\da}{\mathfrak{a}}
\newcommand{\db}{\mathfrak{b}}

\newtheorem{thm}{Theorem}[section]
\newtheorem{lemma}[thm]{Lemma}
\newtheorem{proposition}[thm]{Proposition}
\newtheorem{corollary}[thm]{Corollary}
\newtheorem{claim}[thm]{Claim}

\theoremstyle{definition}
\newtheorem{definition}[thm]{Definition}
\newtheorem{remark}[thm]{Remark}

\newtheorem*{ack}{Acknowledgments}

\begin{document}

\maketitle 

\begin{abstract}
We show that any $n$-dimensional Fano manifold $X$ admitting K\"ahler-Einstein 
metrics satisfies that the anti-canonical volume 
is less than or equal 
to the value $(n+1)^n$. Moreover, the equality holds if and only if $X$ is isomorphic to 
the $n$-dimensional projective space. 
\end{abstract}

\setcounter{tocdepth}{1}
\tableofcontents

\section{Introduction}\label{intro_section}

An $n$-dimensional smooth complex projective variety $X$ is said to be a 
\emph{Fano manifold} if the anti-canonical divisor $-K_X$ is ample. 
If $n\leq 3$, then the anti-canonical volume $((-K_X)^{\cdot n})$ is 
less than or equal to $(n+1)^n$, and the equality holds if and only if $X$ is 
isomorphic to the projective space $\pr^n$ by \cite{isk, MoMu}. 
However, if $n\geq 4$, there exists an $n$-dimensional Fano manifold $X$ 
such that $((-K_X)^{\cdot n})>(n+1)^n$ holds (see \cite[p.\ 128]{IP} for example). 
Recently, Berman and Berndtsson \cite{BB1} conjectured that, if $X$ admits 
\emph{K\"ahler-Einstein metrics}, then the value $((-K_X)^{\cdot n})$ would be 
less than or equal to $(n+1)^n$. In fact, if $X$ is toric, then the conjecture is true 
by \cite[Theorem 1]{BB1} and \cite[Proposition 1.3]{NP}. 
Moreover, Berman and Berndtsson \cite{BB2} proved the above conjecture under the 
assumption that $X$ admits a $\G_m$-action with finite number of fixed points. 

The purpose of this article is to refine the result \cite{BB2} in full generality. 
The following is the main result in this article.

\begin{thm}[Main Theorem]\label{mainthm}
Let $X$ be an $n$-dimensional Fano manifold admitting K\"ahler-Einstein metrics. 
If $((-K_X)^{\cdot n})\geq (n+1)^n$, then $X\simeq\pr^n$.  
\end{thm}

The strategy to prove Theorem \ref{mainthm} is algebraic and is completely 
different from the argument in \cite{BB2}. 
For a Fano manifold $X$, recall that, $X$ admits 
K\"ahler-Einstein metrics if and only if the pair $(X, -K_X)$ is \emph{K-polystable}
(see \cite{tian1, don05, CT, stoppa, mab1, mab2, B, CDS1, CDS2, CDS3, tian2}). 
In \cite{B}, Berman proved the ``only if" direction by viewing 
the slope of the Ding functional (see \cite{din}) 
along a geodesic ray in the space of K\"ahler potentials. 
Berman also treated the case that 
$X$ is a \emph{$\Q$-Fano variety}, that is, a complex projective variety which is 
log terminal and $-K_X$ is an ample $\Q$-Cartier divisor. 
In this article, we heavily use Berman's results \cite{B}. 
In Section \ref{ding_section} of this article, 
we introduce the notions of \emph{Ding polystability} and \emph{Ding semistability}. 
These notions are nothing but interpretations of 
Berman's formula for the slope of the Ding functional. 
The result in \cite[\S 3]{B} shows that, if a $\Q$-Fano variety $X$ admits 
K\"ahler-Einstein metrics, then $X$ is Ding polystable (and also Ding semistable, 
see Theorem \ref{berman_thm}). 
A $\Q$-Fano variety $X$ is said to be Ding semistable if the \emph{Ding invariant} 
$\Ding(\sX, \sL)$ satisfies that $\Ding(\sX, \sL)\geq 0$ for any 
normal test configuration $(\sX, \sL)/\A^1$ of $(X, -rK_X)$ (see Section 
\ref{ding_section} in detail). The key idea for the proof of Theorem \ref{mainthm} 
is constructing specific 
test configurations of $(X, -rK_X)$ from any nonzero proper closed subscheme 
$Z\subset X$ and calculating those Ding invariants and taking the limit. 
The construction of test configurations is similar to the construction in \cite{fjt1, fjt2}. 
We consider a \emph{sequence} of test configurations. 
The following is one of the main consequence of the key idea. 

\begin{thm}[{=Theorem \ref{ding_thm}}]\label{ding_intro_thm}
Let $X$ be a $\Q$-Fano variety. Assume that $X$ is Ding semistable. 
Take any nonempty proper closed subscheme $\emptyset\neq Z\subsetneq X$ 
corresponds to an ideal sheaf $0\neq I_Z\subsetneq\sO_X$. 
Let $\sigma\colon\hat{X}\to X$ be the blowup along $Z$, let $F\subset\hat{X}$ be 
the Cartier divisor defined by the equation $\sO_{\hat{X}}(-F)=I_Z\cdot\sO_{\hat{X}}$. 
Then we have $\beta(Z)\geq 0$, where 
\[
\beta(Z):=\lct(X;I_Z)\cdot\vol_X(-K_X)-
\int_0^\infty\vol_{\hat{X}}\left(\sigma^*(-K_X)-xF\right)dx.
\]
Note that, $\vol$ is the volume function $($see Definition \ref{volume_dfn}$)$, 
and $\lct(X; I_Z)$ is the log canonical threshold of $I_Z$ with respects to 
$X$ $($see Definition \ref{lct_dfn}$)$. 
\end{thm}

More generally, we construct a sequence of test configurations from 
\emph{filtered linear series} in Section \ref{tc_filt_section}. 
From Theorem \ref{ding_intro_thm}, we can immediately show the following corollary. 

\begin{corollary}[{see Theorem \ref{Pn_thm}}]\label{intro_cor}
Let $X$ be an $n$-dimensional $\Q$-Fano variety. Assume that 
$X$ is Ding semistable. Then we have 
$((-K_X)^{\cdot n})\leq (n+1)^n$. 
\end{corollary}

Theorem \ref{mainthm} is immediately obtained by Corollary \ref{intro_cor} and 
a description of Seshadri constants (Theorem \ref{sesh_thm}), 
together with the results \cite{CMSB} and \cite{kebekus}. 
For detail, see Section \ref{proof_section}.

The article is organized as follows. 
In Section \ref{prelim_section}, we recall the notions of the volume functions, 
Seshadri constants, log canonical thresholds and K-stability. 
We characterize Seshadri constants in terms of the 
volume function in Theorem \ref{sesh_thm}. 
The theorem is important in order to characterize the projective space. 
In Section \ref{ding_section}, we recall Berman's result \cite{B}. 
We introduce the notions of Ding invariants, Ding polystability and Ding semistability. 
Section \ref{ding_subscheme_section} is the core of this article. 
In Section \ref{saturation_section}, we consider a general theory of the saturation of 
filtered linear series. In Section \ref{tc_filt_section}, we construct a sequence of 
semi test configurations from given filtered linear series. 
The construction is similar to the one in \cite{sz}. Our construction enables us to 
calculate (a kind of) the limit of those Ding invariants via the saturation of the given 
filtration. See Theorem \ref{filt_thm} in detail. 
In Section \ref{ding_sub_section}, 
motivated by the work of Ross and Thomas \cite{RT}, we consider specific test configurations obtained by the natural filtered linear series coming from fixed 
closed subschemes. By taking the limit of those Ding invariants, 
we get Theorem \ref{ding_thm}. 
In Section \ref{proof_section}, we prove Theorem \ref{mainthm}. 
This is an immediate consequence of previous sections.

\begin{ack}
The author thanks Doctor Yuji Odaka, who introduced him the importance 
of \cite[\S 3]{B} and helped him to deduce Proposition \ref{ding_flag_prop}, 
and Professor Robert Berman, who gave him comments related to \cite{B}. 
The author is partially supported by a JSPS Fellowship for Young Scientists. 
\end{ack}

Throughout this paper, we work in the category of algebraic (separated and of 
finite type) scheme over the complex number field $\C$. 
A \emph{variety} means a reduced and irreducible algebraic scheme. 
For a projective surjective morphism $\alpha\colon\sX\to C$ with $\sX$ a normal 
variety and $C$ a smooth curve, let $K_{\sX/C}:=K_{\sX}-\alpha^*K_C$ be the relative 
canonical divisor. Moreover, for a closed point $t\in C$, let $\sX_t$ be the 
scheme-theoretic fiber of $\alpha$ at $t\in C$. 
For a $\Q$-Fano variety $X$, $\omega$ is said to be a 
\emph{K\"ahler-Einstein metric} on $X$ if $\omega$ is a K\"ahler-Einstein metric 
on the smooth locus $X^{\SM}$ of $X$ and the volume of $\omega$ on $X^{\SM}$ 
coincides with the value $((-K_X)^{\cdot n})$ (see \cite{BBEGZ, B} for detail). 

For any $c\in\R$, let $\lfloor c\rfloor\in\Z$ 
be the biggest integer which is not bigger than 
$c$ and let $\lceil c\rceil\in\Z$ 
be the smallest integer which is not less than $c$.

\section{Preliminaries}\label{prelim_section}

In this section, we recall some basic definitions and see those properties. 

\subsection{The volumes of divisors}\label{volume_section}

\begin{definition}[{see \cite{L1, L2}}]\label{volume_dfn}
Let $X$ be an $n$-dimensional projective variety. 
For a Cartier divisor $L$ on $X$, we set 
\[
\vol_X(L):=\limsup_{k\to\infty}\frac{h^0(X, \sO_X(kL))}{k^n/n!}.
\]
We know that the limsup computing $\vol_X(L)$ is actually a limit 
(see \cite[Example 11.4.7]{L2}). If $L$ and $L'$ are numerically equivalent, 
then $\vol_X(L)=\vol_X(L')$ (see \cite[Proposition 2.2.41]{L1}). Moreover, 
we can extend uniquely to a continuous function
\[
\vol_X\colon \ND(X)\to\R_{\geq 0}
\]
(see \cite[Corollary 2.2.45]{L1}).
\end{definition}

\subsection{Seshadri constants, pseudo-effective thresholds}\label{sesh_section}

\begin{definition}\label{sesh_psef_dfn}
Let $X$ be a projective variety, $L$ be an ample $\Q$-divisor 
on $X$, $\emptyset\neq Z\subsetneq X$ be a nonempty proper subscheme 
corresponds to an ideal sheaf $0\neq I_Z\subsetneq\sO_X$, 
$\sigma\colon\hat{X}\to X$ be the blowup along $Z$, and $F\subset\hat{X}$ be the 
Cartier divisor defined by the equation $\sO_{\hat{X}}(-F)=I_Z\cdot\sO_{\hat{X}}$. 
\begin{enumerate}
\renewcommand{\theenumi}{\arabic{enumi}}
\renewcommand{\labelenumi}{(\theenumi)}
\item\label{sesh_psef_dfn1}
The \emph{Seshadri constant $\varepsilon_Z(L)$ of $L$ along $Z$} is defined by
\[
\varepsilon_Z(L):=\sup\{x\in\R_{>0}\,|\,\sigma^*L-xF: \text{ ample}\}.
\]
\item\label{sesh_psef_dfn2}
The \emph{pseudo-effective threshold $\tau_Z(L)$ of $L$ along $Z$} is defined by
\[
\tau_Z(L):=\sup\{x\in\R_{>0}\,|\,\sigma^*L-xF: \text{ big}\}.
\]
\end{enumerate}
If $X$ is a $\Q$-Fano variety, then we write $\varepsilon_Z:=\varepsilon_Z(-K_X)$ and 
$\tau_Z:=\tau_Z(-K_X)$ for simplicity. 
\end{definition}

\begin{thm}\label{sesh_thm}
Let $X$ be an $n$-dimensional projective variety with $n\geq 2$, $L$ be an ample 
$\Q$-divisor on $X$, $p\in X$ be a smooth closed point, $\sigma\colon\hat{X}\to X$ 
be the blowup along $p$, and $F\subset\hat{X}$ be the exceptional divisor of $\sigma$. 
\begin{enumerate}
\renewcommand{\theenumi}{\arabic{enumi}}
\renewcommand{\labelenumi}{(\theenumi)}
\item\label{sesh_thm1}
For any $x\in\R_{\geq 0}$, we have
\[
\vol_{\hat{X}}(\sigma^*L-xF)\geq((\sigma^*L-xF)^{\cdot n})=(L^{\cdot n})-x^n.
\]
\item\label{sesh_thm2}
Set $\Lambda_p(L):=\{x\in\R_{\geq 0}\,|\,\vol_{\hat{X}}(\sigma^*L-xF)
=((\sigma^*L-xF)^{\cdot n})\}$. Then we have
\[
\varepsilon_p(L)=\max\{x\in\R_{\geq 0}\,|\,y\in\Lambda_p(L) \text{ for all }y\in[0,x]\}.
\]
\end{enumerate}
\end{thm}

\begin{proof}
Take any $k\in\Z_{>0}$ such that $kL$ is Cartier. For any $j\in\Z_{>0}$, we have 
\begin{eqnarray*}
h^0\left(jF, \sigma^*\sO_X(kL)|_{jF}\right) & = & \sum_{l=0}^{j-1}\binom{n-1+l}{n-1}
=\binom{n-1+j}{n},\\
h^i\left(jF, \sigma^*\sO_X(kL)|_{jF}\right) & = & 0 \,\,\,(\text{if } i>0), 
\end{eqnarray*}
since we have exact sequences
\[
0\to \sO_{\pr^{n-1}}(l)\to \sigma^*\sO_X(kL)|_{(l+1)F} \to \sigma^*\sO_X(kL)|_{lF} \to 0
\]
for all $1\leq l\leq j-1$. 

\eqref{sesh_thm1}
We can assume that $x\in\Q_{>0}$ since the function $\vol_{\hat{X}}(\sigma^*L-xF)$ 
is continuous. Take any sufficiently large 
$k\in\Z_{>0}$ with $kx\in\Z_{>0}$ and $kL$ Cartier. 
Since \[
H^1\left(\hat{X}, \sigma^*\sO_X(kL)\right)\simeq H^1\left(X, \sO_X(kL)\right)=0, 
\]
we get the following exact sequence: 
\begin{eqnarray*}
0\to H^0\left(\hat{X}, \sO_{\hat{X}}(\sigma^*(kL)-kxF)\right)
\to H^0\left(\hat{X}, \sigma^*\sO_X(kL)\right)\\
\to H^0\left(kxF, \sigma^*\sO_X(kL)|_{kxF}\right)
\to H^1\left(\hat{X}, \sO_{\hat{X}}(\sigma^*(kL)-kxF)\right)
\to 0.
\end{eqnarray*}
Thus we have
\begin{eqnarray*}
&&h^0\left(\hat{X}, \sO_{\hat{X}}(\sigma^*(kL)-kxF)\right)\\
&\geq& h^0\left(X, \sO_X(kL)\right)-\binom{n-1+kx}{n}
=\frac{(L^{\cdot n})-x^n}{n!}k^n+o(k^n).
\end{eqnarray*}

\eqref{sesh_thm2}
Let $a$ be the right-hand side of the equation in \eqref{sesh_thm2}. 
For any nef divisor $M$, the volume of $M$ is equal to the self intersection number. 
Thus the inequality $\varepsilon_p(L)\leq a$ is obvious. In particular, we have $a>0$. 
Take any $\varepsilon\in\R_{>0}$ such that $a-\varepsilon\in\Q_{>0}$. 
It is enough to show that $\sigma^*L-(a-\varepsilon)F$ is ample in order to 
show the inequality $\varepsilon_p(L)\geq a$. 
Fix $\delta\in\Q_{>0}$ such that $\delta<\varepsilon_p(L)$, that is, 
$\sigma^*L-\delta F$ is ample. Take any rational number $t$ with 
\[
0\leq t<\min\left\{1, \frac{a-\varepsilon}{\delta}, \frac{\varepsilon}{a-\delta}\right\}, 
\]
and set $x_t:=(a-\varepsilon-t\delta)/(1-t)$. We note that $x_t\in(0, a)\cap\Q$. 
Moreover, we have
\[
\sigma^*L-(a-\varepsilon)F-t(\sigma^*L-\delta F)=(1-t)(\sigma^*L-x_t F). 
\]
Take any sufficiently large $k\in\Z_{>0}$ with $kx_t\in\Z_{>0}$ and $kL$ Cartier. 
Then, from the exact sequence 
\[
0\to \sO_{\hat{X}}\left(\sigma^*(kL)-kx_tF\right)
\to \sO_{\hat{X}}\left(\sigma^*(kL)\right)
\to \sigma^*\sO_X(kL)|_{kx_tF}\to 0
\]
and the previous arguments, we have 
\begin{eqnarray*}
&&\limsup_{k}\frac{h^1\left(\hat{X}, \sO_{\hat{X}}(\sigma^*(kL)-kx_tF)\right)}{k^n/n!}\\
&=&\limsup_{k}\Biggl(\frac{h^0\left(\hat{X}, \sO_{\hat{X}}
(\sigma^*(kL)-kx_tF)\right)}{k^n/n!}\\
&&+\frac{h^0\left(kx_tF, 
\sigma^*\sO_X(kL)|_{kx_tF}\right)}{k^n/n!}
-\frac{h^0(X, \sO_X(kL))}{k^n/n!}\Biggr)\\
&=&\vol_{\hat{X}}(\sigma^*L-x_tF)+x_t^n-(L^{\cdot n})=0
\end{eqnarray*}
since $x_t\in\Lambda_p(L)$. Similarly, we have 
\[
h^i\left(\hat{X}, \sO_{\hat{X}}(\sigma^*(kL)-kx_tF)\right)=
h^i\left(\hat{X}, \sO_{\hat{X}}(\sigma^*(kL))\right)=0
\]
for any $i\geq 2$. Thus, by \cite[\S 2.3 and Theorem 4.1]{dFKL}, 
$\sigma^*L-(a-\varepsilon)F$ is ample. Therefore the assertion follows. 
\end{proof}

\subsection{Log canonical thresholds}\label{lct_section}

\begin{definition}\label{lc_dfn}
\begin{enumerate}
\renewcommand{\theenumi}{\arabic{enumi}}
\renewcommand{\labelenumi}{(\theenumi)}
\item\label{lc_dfn1}
Let $(Y, \Delta)$ be a pair such that $Y$ is a normal variety and 
$\Delta$ is a (possibly non-effective) $\R$-divisor on $Y$ such that 
$K_Y+\Delta$ is $\R$-Cartier. 
The pair $(Y, \Delta)$ is said to be \emph{sub log canonical} if 
$a(E, Y, \Delta)\geq -1$ holds 
for any proper birational morphism $\phi\colon\tilde{Y}\to Y$ with $\tilde{Y}$ normal 
and for any prime divisor $E$ on $\tilde{Y}$, where 
$a(E, Y, \Delta):=\ord_E(K_{\tilde{Y}}-\phi^*(K_Y+\Delta))$. 
\item\label{lc_dfn2}
Let $Y$ be a variety which is log terminal, 
$\da_1,\dots,\da_l\subset\sO_Y$ be coherent nonzero ideal sheaves, and 
$c_1,\dots,c_l$ be (possibly negative) real numbers. 
The pair $(Y, \da_1^{\cdot c_1}\cdots\da_l^{\cdot c_l})$ is said to be 
\emph{sub log canonical} if $a(E, Y, \da_1^{\cdot c_1}\cdots\da_l^{\cdot c_l})\geq -1$ 
holds for any proper birational morphism 
$\phi\colon\tilde{Y}\to Y$ with $\tilde{Y}$ normal 
and for any prime divisor $E$ on $\tilde{Y}$, where $a(E, Y, 
\da_1^{\cdot c_1}\cdots\da_l^{\cdot c_l}):=\ord_E(K_{\tilde{Y}}-\phi^*K_Y)
-\sum_{i=1}^lc_i\cdot\ord_E(\da_i)$. 
\item\label{lc_dfn3}
Let $Y$ be a variety which is log terminal, $r_0\in\Z_{>0}$, $\{\da_r\}_{r\geq r_0}$ be a 
graded family of coherent ideal sheaves on $Y$, that is, $\da_r\cdot\da_{r'}\subset
\da_{r+r'}$ holds for any $r$, $r'\geq r_0$, $\db_1,\dots,\db_l\subset\sO_Y$ be 
coherent nonzero ideal sheaves, $c_1,\dots,c_l\in\R$ and $c\in\R_{>0}$. 
The pair $(Y, \da_\bullet^{\cdot c}\cdot\db_1^{\cdot c_1}\cdots\db_l^{\cdot c_l})$
is said to be \emph{sub log canonical} if 
$a(E, Y, \da_\bullet^{\cdot c}\cdot\db_1^{\cdot c_1}\cdots\db_l^{\cdot c_l})\geq -1$
holds for any proper birational morphism 
$\phi\colon\tilde{Y}\to Y$ with $\tilde{Y}$ normal 
and for any prime divisor $E$ on $\tilde{Y}$, where 
$a(E, Y, \da_\bullet^{\cdot c}\cdot\db_1^{\cdot c_1}\cdots\db_l^{\cdot c_l})$
is defined by the value
\[
\ord_E(K_{\tilde{Y}}-\phi^*K_Y)-\sum_{i=1}^l c_i\cdot\ord_E(\db_i)
-\liminf_{r\to\infty}\frac{c\cdot\ord_E(\da_r)}{r}.
\]
\end{enumerate}
\end{definition}

\begin{lemma}\label{gr_lc_lem}
Let $Y$ be a variety which is log terminal, $r_0\in\Z_{>0}$, $\{\da_r\}_{r\geq r_0}$ be a 
graded family of coherent ideal sheaves on $Y$, $\db\subset\sO_Y$ be a coherent 
nonzero ideal sheaf, $c\in\R_{>0}$ and $a\in\R$. 
\begin{enumerate}
\renewcommand{\theenumi}{\arabic{enumi}}
\renewcommand{\labelenumi}{(\theenumi)}
\item\label{gr_lc_lem1}
Assume that there exists a sequence $\{a_r\}_{r\geq r_0}$ with $\lim_{r\to\infty}a_r=a$ 
and the pair $(Y, \da_r^{\cdot(c/r)}\cdot\db^{\cdot a_r})$ is 
sub log canonical for any sufficiently divisible $r\gg 0$. Then the pair 
$(Y, \da_\bullet^{\cdot c}\cdot\db^{\cdot a})$ is sub log canonical.
\item\label{gr_lc_lem2}
Assume that there exists a coherent ideal sheaf $I\subset\sO_Y$ such that 
$\da_r\subset I^r$ for any $r\geq r_0$ 
and the pair $(Y, \da_\bullet^{\cdot c}\cdot\db^{\cdot a})$ 
is sub log canonical. Then the pair 
$(Y, I^{\cdot c}\cdot\db^{\cdot a})$ is sub log canonical. 
\end{enumerate}
\end{lemma}

\begin{proof}
Take any proper birational morphism $\phi\colon\tilde{Y}\to Y$ with $\tilde{Y}$ normal 
and a prime divisor $E$ on $\tilde{Y}$. 
For any $r\geq r_0$ and $k\in\Z_{>0}$, we have 
\[
\frac{1}{kr}\ord_E(\da_{kr})\leq\frac{1}{kr}\ord_E(\da_r^k)=\frac{1}{k}\ord_E(\da_r). 
\]
Thus we have 
\[
\liminf_{r\to\infty}\frac{c\cdot\ord_E(\da_r)}{r}=
\liminf_{r\to\infty}\frac{c\cdot\ord_E(\da_{kr})}{kr}
\]
for any $k\in\Z_{>0}$. 

\eqref{gr_lc_lem1}
By assumption, for any sufficiently divisible $r\gg 0$, 
\[
-1\leq\ord_E(K_{\tilde{Y}}-\phi^*K_Y)-\frac{c\cdot\ord_E(\da_r)}{r}-a_r\cdot\ord_E(\db)
\]
holds. By taking $\limsup_{r\to\infty}$, we have 
$-1\leq a(E, Y, \da_\bullet^{\cdot c}\cdot\db^{\cdot a})$. 

\eqref{gr_lc_lem2}
For any $r\geq r_0$, we have $cr^{-1}\cdot\ord_E(\da_r)\geq c\cdot\ord_E(I)$. 
Thus we get the inequality 
$-1\leq a(E, Y, I^{\cdot c}\cdot\db^{\cdot a})$. 
\end{proof}

\begin{definition}\label{lct_dfn}
\begin{enumerate}
\renewcommand{\theenumi}{\arabic{enumi}}
\renewcommand{\labelenumi}{(\theenumi)}
\item\label{lct_dfn1}
Let $(Y, \Delta)$ be a pair as in Definition \ref{lc_dfn} \eqref{lc_dfn1} and 
$B$ be a nonzero effective $\R$-Cartier divisor on $Y$. 
The \emph{log canonical threshold $\lct(Y, \Delta; B)$ of $B$ with respects to 
$(Y, \Delta)$} is defined by the following: 
\begin{itemize}
\item
If the pair $(Y, \Delta+cB)$ is not sub log canonical for any $c\in\R$, then we set 
$\lct(Y, \Delta; B):=-\infty$. 
\item
Otherwise, we set 
\[
\lct(Y, \Delta; B):=\sup\{c\in\R\,|\,(Y, \Delta+cB):\text{ sub log canonical}\}.
\]
\end{itemize}
\item\label{lct_dfn2}
Let $(Y, \da_1^{\cdot c_1}\cdots\da_l^{\cdot c_l})$ be a pair 
as in Definition \ref{lc_dfn} \eqref{lc_dfn2} and 
$0\neq\db\subsetneq\sO_Y$ be a coherent ideal sheaf. 
The \emph{log canonical threshold 
$\lct(Y, \da_1^{\cdot c_1}\cdots\da_l^{\cdot c_l}; \db)$ of $\db$ with respects to 
$(Y, \da_1^{\cdot c_1}\cdots\da_l^{\cdot c_l})$} is defined by the following: 
\begin{itemize}
\item
If the pair $(Y, \da_1^{\cdot c_1}\cdots\da_l^{\cdot c_l}\cdot\db^{\cdot c})$ 
is not sub log canonical for any $c\in\R$, then we set 
$\lct(Y, \da_1^{\cdot c_1}\cdots\da_l^{\cdot c_l}; \db):=-\infty$. 
\item
Otherwise, we set 
\begin{eqnarray*}
&&\lct(Y, \da_1^{\cdot c_1}\cdots\da_l^{\cdot c_l}; \db)\\
&:=&
\sup\{c\in\R\,|\,(Y, \da_1^{\cdot c_1}\cdots\da_l^{\cdot c_l}\cdot\db^{\cdot c}):\text{ sub log canonical}\}.
\end{eqnarray*}
\end{itemize}
Moreover, if $l=1$ and $\da_1=\sO_Y$, then we write 
$\lct(Y; \db):=\lct(Y, \da_1^{\cdot c_1}; \db)$ for simplicity. 
\end{enumerate}
\end{definition}

\subsection{K-stability}\label{K_section}

\begin{definition}[{\cite{tian1, don, RT, odk, LX}}]\label{K_dfn}
Let $X$ be an $n$-dimensional $\Q$-Fano variety. 
\begin{enumerate}
\renewcommand{\theenumi}{\arabic{enumi}}
\renewcommand{\labelenumi}{(\theenumi)}
\item\label{K_dfn1}
Let $r\in\Z_{>0}$ such that $-rK_X$ is Cartier. 
A \emph{test configuration} (resp.\ a \emph{semi test configuration}) 
$(\sX, \sL)/\A^1$ \emph{of} $(X, -rK_X)$ consists of the following data:
\begin{itemize}
\item
a variety $\sX$ such that admitting $\G_m$-action and the morphism 
$\alpha\colon\sX\to\A^1$ is $\G_m$-equivariant, 
where the action $\G_m\times\A^1\to\A^1$ is given by $(a, t)\mapsto at$, and
\item
a $\G_m$-equivariant $\alpha$-ample (resp.\ $\alpha$-semiample) line bundle 
$\sL$ on $\sX$ such that $(\sX, \sL)|_{\alpha^{-1}(\A^1\setminus\{0\})}$ is 
$\G_m$-equivariantly isomorphic to $(X, \sO_X(-rK_X))\times(\A^1\setminus\{0\})$ 
with the natural $\G_m$-action.
\end{itemize}
Moreover, if $\sX$ is normal in addition, then we call the $(\sX, \sL)/\A^1$ 
a \emph{normal test configuration} (resp.\ a \emph{normal semi test configuration})
\emph{of} $(X, -rK_X)$. 
\item\label{K_dfn2}
Assume that $(\sX, \sL)/\A^1$ is a normal semi test configuration of $(X, -rK_X)$. 
Let $\alpha\colon(\bar{\sX}, \bar{\sL})\to\pr^1$ be the natural equivariant 
compactification of $(\sX, \sL)\to\A^1$ induced by the compactification 
$\A^1\subset\pr^1$. The \emph{Donaldson-Futaki invariant} 
$\DF(\sX, \sL)$ of $(\sX, \sL)/\A^1$ is defined by 
\[
\DF(\sX, \sL):=\frac{1}{(n+1)((-K_X)^{\cdot n})}
\left(\frac{n}{r^{n+1}}(\bar{\sL}^{\cdot n+1})+
\frac{n+1}{r^n}(\bar{\sL}^{\cdot n}\cdot K_{\bar{\sX}/\pr^1})\right).
\]
\item\label{K_dfn3}
\begin{itemize}
\item
The pair $(X, -K_X)$ is called \emph{K-semistable} if $\DF(\sX, \sL)\geq 0$ for 
any normal test configuration $(\sX, \sL)/\A^1$ of $(X, -rK_X)$. 
\item
The pair $(X, -K_X)$ is called \emph{K-polystable} if $\DF(\sX, \sL)\geq 0$ for 
any normal test configuration $(\sX, \sL)/\A^1$ of $(X, -rK_X)$, and the equality holds 
only if $\sX\simeq X\times\A^1$.  
\item
The pair $(X, -K_X)$ is called \emph{K-stable} if $\DF(\sX, \sL)\geq 0$ for 
any normal test configuration $(\sX, \sL)/\A^1$ of $(X, -rK_X)$, and the equality holds 
only if the pair $(\sX, \sL)$ is trivial, that is, the pair $(\sX, \sL)$ is 
$\G_m$-equivariantly isomorphic to the pair 
$(X\times\A^1, \sO_{X\times\A^1}(-rK_{X\times\A^1/\A^1}))$ 
with the natural $\G_m$-action. 
\end{itemize}
\end{enumerate}
\end{definition}

\section{Ding polystability}\label{ding_section}

We recall the theory in \cite[\S 3]{B}. The author learned the theory from Odaka. 

\begin{definition}[{see \cite[\S 3]{B}}]\label{ding_dfn}
Let $X$ be an $n$-dimensional $\Q$-Fano variety. 
\begin{enumerate}
\renewcommand{\theenumi}{\arabic{enumi}}
\renewcommand{\labelenumi}{(\theenumi)}
\item\label{ding_dfn1}
Let $(\sX, \sL)/\A^1$ be a normal semi test configuration of $(X, -rK_X)$ and 
$(\bar{\sX}, \bar{\sL})/\pr^1$ be its natural compactification 
as in Definition \ref{K_dfn} \eqref{K_dfn2}. 
\begin{enumerate}
\renewcommand{\theenumii}{\roman{enumii}}
\renewcommand{\labelenumii}{(\theenumii)}
\item\label{ding_dfn11}
Let $D_{(\sX, \sL)}$ be the $\Q$-divisor on $\sX$ such that the following conditions 
are satisfied: 
\begin{itemize}
\item
The support $\Supp D_{(\sX, \sL)}$ is contained in $\sX_0$. (Note that 
$\sX_0$ is the fiber of $\sX\to\A^1$ at $0\in\A^1$.) 
\item 
The divisor $-rD_{(\sX, \sL)}$ is a $\Z$-divisor corresponds to the divisorial sheaf 
$\bar{\sL}(rK_{\bar{\sX}/\pr^1})$. (Thus the divisor 
$-r(K_{\bar{\sX}/\pr^1}+D_{(\sX, \sL)})$ is a Cartier divisor corresponds to 
$\bar{\sL}$.)
\end{itemize}
Since the divisorial sheaf $\bar{\sL}(rK_{\bar{\sX}/\pr^1})$ is trivial on 
$\bar{\sX}\setminus\sX_0$, the $D_{(\sX, \sL)}$ exists and is unique. 
\item\label{ding_dfn12}
The \emph{Ding invariant} $\Ding(\sX, \sL)$ of $(\sX, \sL)/\A^1$ is defined by 
\[
\Ding(\sX, \sL):=\frac{-(\bar{\sL}^{\cdot n+1})}{(n+1)r^{n+1}((-K_X)^{\cdot n})}
-\left(1-\lct(\sX, D_{(\sX, \sL)}; \sX_0)\right).
\]
\end{enumerate}
\item\label{ding_dfn2}
\begin{itemize}
\item
$X$ is called \emph{Ding semistable} if $\Ding(\sX, \sL)\geq 0$ for any normal 
test configuration $(\sX, \sL)/\A^1$ of $(X, -rK_X)$. 
\item
$X$ is called \emph{Ding polystable} if 
\begin{itemize}
\item
$X$ is Ding semistable, and
\item
if $(\sX, \sL)$ is a normal test configuration of $(X, -rK_X)$ which satisfies that 
$\sL\simeq\sO_{\sX}(-rK_{\sX/\A^1})$, $\sX_0$ is log terminal and $\Ding(\sX, \sL)=0$, 
then $\sX\simeq X\times\A^1$. 
\end{itemize}
\end{itemize}
\end{enumerate}
\end{definition}

The following is a theorem of Berman. 

\begin{thm}[{\cite{B}}]\label{berman_thm}
Let $X$ be a $\Q$-Fano variety. 
\begin{enumerate}
\renewcommand{\theenumi}{\arabic{enumi}}
\renewcommand{\labelenumi}{(\theenumi)}
\item\label{berman_thm1}
If $X$ admits K\"ahler-Einstein metrics, then $X$ is Ding polystable. 
\item\label{berman_thm2}
For any normal test configuration $(\sX, \sL)/\A^1$ of $(X, -rK_X)$, we have 
$\DF(\sX,\sL)\geq \Ding(\sX, \sL)$. Moreover, the equality holds if and only if 
\begin{itemize}
\item
$\sL\simeq\sO_{\sX}(-rK_{\sX/\A^1})$, and
\item
the pair $(\sX, \sX_0)$ is log canonical.
\end{itemize}
\end{enumerate}
\end{thm}

\begin{proof}
We repeat the proof in \cite[\S 3]{B} for the reader's convenience. 
Pick any normal test configuration $(\sX, \sL)/\A^1$ of $(X, -rK_X)$. 
Set $n:=\dim X$ and 
\begin{eqnarray*}
q(\sX, \sL) & := & \DF(\sX, \sL)-\Ding(\sX, \sL)\\
 & = & 1-\lct(\sX, D_{(\sX, \sL)}; \sX_0)-
\frac{(\bar{\sL}^{\cdot n}\cdot D_{(\sX, \sL)})}{r^n((-K_X)^{\cdot n})}.
\end{eqnarray*}

\eqref{berman_thm1}
Let $\gamma\colon\sX'\to\sX$ be a $\G_m$-equivariant log resolution 
of the pair $(\sX, \sX_0)$ and let $\gamma\colon\bar{\sX}'\to \bar{\sX}$ be its natural 
compactification. Since $\bar{\sX}\setminus\sX_0$ is log terminal, if we set 
$D^*:=-K_{\bar{\sX}'\setminus\sX'_0}+\gamma^*K_{\bar{\sX}\setminus\sX_0}$ then 
any coefficient of $D^*$ is strictly smaller than one. 
Let $\Delta'$ be the $\Q$-divisor on $\bar{\sX}'$ 
such that the following conditions are satisfied: 
\begin{itemize}
\item
The support $\Supp(\Delta'+\bar{D}^*)$ is contained in $\sX'_0$, where 
$\bar{D}^*$ is the closure of $D^*$ in $\bar{\sX}'$.
\item
The divisor $r\Delta'$ is a $\Z$-divisor and corresponds to 
$\gamma^*\bar{\sL}(rK_{\bar{\sX}'/\pr^1})$.
\end{itemize}
Let 
\[
\sX'_0=\sum_{i\in I}m'_iE'_i,\quad\Delta'+\bar{D}^*=\sum_{i\in I}c'_iE'_i
\]
be the irreducible decompositions. By construction, we have 
\[
\gamma^*(K_{\bar{\sX}}+D_{(\sX, \sL)}+c\sX_0)=K_{\bar{\sX}'}-\Delta'+c\sX'_0
\]
for any $c\in\R$. Thus we have 
\[
\lct(\sX, D_{(\sX, \sL)}; \sX_0)=\lct(\sX', -\Delta'; \sX'_0)=
\min_{i\in I}\left\{\frac{1+c'_i}{m'_i}\right\}.
\]
Moreover, we have 
\[
-(\bar{\sL}^{\cdot n}\cdot D_{(\sX, \sL)})
=(\gamma^*\bar{\sL}^{\cdot n}\cdot\Delta'+\bar{D}^*)
=\sum_{i\in I}c'_i(\gamma^*\bar{\sL}^{\cdot n}\cdot E'_i)
\]
since $\gamma_*(\Delta'+\bar{D}^*)=-D_{(\sX, \sL)}$ holds. Therefore, 
\[
q(\sX, \sL)=\max_{i\in I}\left\{\frac{m'_i-1-c'_i}{m'_i}\right\}
+\frac{1}{r^n((-K_X)^{\cdot n})}\sum_{i\in I}c'_i(\gamma^*\bar{\sL}^{\cdot n}\cdot E'_i)
\]
holds. The equation is nothing but Formula (3.30) in \cite{B}. 
Hence, if $X$ admits K\"ahler-Einstein metrics, then $\Ding(\sX, \sL)\geq 0$ holds 
by \cite[Theorem 3.11 and Formula (3.2)]{B} (see also \cite{Bern} and
\cite[Formula (6.5)]{BBGZ}). 
If we further assume that 
$\sL\simeq\sO_{\sX}(-rK_{\sX/\A^1})$, $\sX_0$ is log terminal and $\Ding(\sX, \sL)=0$, 
then $D_{(\sX, \sL)}=c\sX_0$ for some $c\in\Q$ 
and $\lct(\sX, D_{(\sX, \sL)}; \sX_0)=1-c$. 
This implies that $\DF(\sX, \sL)=\Ding(\sX, \sL)$ 
since $q(\sX, \sL)=0$ holds. Hence $\sX\simeq X\times\A^1$ by 
\cite[Theorem 1.1]{B} (more precisely, by \cite[Proposition 3.5]{B}). 
Thus $X$ is Ding polystable. 

\eqref{berman_thm2} (See \cite[Proof of Theorem 3.11]{B}.)
Let 
\[
\sX_0=\sum_{i\in J}m_iE_i,\quad -D_{(\sX, \sL)}=\sum_{i\in J}c_iE_i
\]
be the irreducible decompositions. Note that 
\[
q(\sX, \sL)=\frac{1}{r^n((-K_X)^{\cdot n})}
\left(\bar{\sL}^{\cdot n}\cdot (1-\lct(\sX, D_{(\sX, \sL)}; \sX_0))\sX_0
-D_{(\sX, \sL)}\right).
\]
Since 
\[
1-\lct(\sX, D_{(\sX, \sL)}; \sX_0)\geq\max_{i\in J}\left\{\frac{m_i-1-c_i}{m_i}\right\},
\]
we have 
\begin{eqnarray*}
&&(1-\lct(\sX, D_{(\sX, \sL)}; \sX_0))\sX_0-D_{(\sX, \sL)}\\
&\geq&
\sum_{i\in J}\left(\frac{m_i-1-c_i}{m_i}\cdot m_i+c_i\right)E_i=\sum_{i\in J}(m_i-1)E_i
\geq 0
\end{eqnarray*}
Since $\bar{\sL}$ is $\alpha$-ample, we get $q(\sX, \sL)\geq 0$. 
Moreover, $q(\sX, \sL)=0$ holds if and only if $\sX_0$ is reduced and 
$D_{(\sX, \sL)}=(1-\lct(\sX, D_{(\sX, \sL)}; \sX_0))\sX_0$ holds. 
Thus we get the assertion. 
\end{proof}

\begin{remark}\label{berman_remark}
From Theorem \ref{berman_thm} and \cite[Corollary 1]{LX} (see \cite{B}), 
if a $\Q$-Fano variety $X$ is Ding semistable (resp.\ Ding polystable), then 
the pair $(X, -K_X)$ is K-semistable (resp.\ K-polystable). Thus, 
by \cite{CDS1, CDS2, CDS3, tian2}, if $X$ is a Fano manifold, then  
the following three conditions are equivalent: 
\begin{itemize}
\item
$X$ admits K\"ahler-Einstein metrics.
\item
$X$ is Ding polystable.
\item
$(X, -K_X)$ is K-polystable. 
\end{itemize}
\end{remark}

\begin{lemma}\label{crepant_lem}
Let $X$ be a $\Q$-Fano variety and $\gamma\colon(\sY, \gamma^*\sL)\to
(\sX, \sL)$ be a $\G_m$-equivariant birational morphism between normal 
semi test configurations of $(X, -rK_X)$. Then 
$\Ding(\sX, \sL)=\Ding(\sY, \gamma^*\sL)$ holds. 
\end{lemma}

\begin{proof}
Since $K_{\bar{\sY}}+D_{(\sY, \gamma^*\sL)}=\gamma^*(K_{\bar{\sX}}+D_{(\sX, \sL)})$, 
we have 
\[
\lct(\sX, D_{(\sX, \sL)}; \sX_0)=\lct(\sY, D_{(\sY, \gamma^*\sL)}; \sY_0).
\] 
Thus the assertion follows immediately.
\end{proof}

\begin{proposition}\label{ding_flag_prop}
Let $X$ be an $n$-dimensional $\Q$-Fano variety which is Ding semistable, 
let $r$ be a positive integer such that $-rK_X$ is Cartier, 
let $I_M\subset\dots\subset I_1\subset\sO_X$ be a sequence of coherent ideal sheaves, 
let $\sI:=I_M+I_{M-1}t^1+\dots+I_1t^{M-1}+(t^M)\subset\sO_{X\times\A^1_t}$, let 
$\Pi\colon\sX\to X\times\A^1$ be the blowup along $\sI$, let 
$E\subset\sX$ be the Cartier divisor defined by $\sO_\sX(-E)=\sI\cdot\sO_\sX$, and 
let $\sL:=\Pi^*\sO_{X\times\A^1}(-rK_{X\times \A^1/\A^1})\otimes\sO_\sX(-E)$. 
Assume that $\sL$ is semiample over $\A^1$. Then $(\sX, \sL)/\A^1$ is naturally 
seen as a $($possibly non-normal$)$ semi test configuration of $(X, -rK_X)$. 
Under these conditions, the pair 
$(X\times\A^1_t, \sI^{\cdot (1/r)}\cdot (t)^{\cdot d})$ must be sub log canonical, 
where 
\[
d:=1+\frac{(\bar{\sL}^{\cdot n+1})}{(n+1)r^{n+1}((-K_X)^{\cdot n})}. 
\]
Moreover, we have the equality 
\[
(\bar{\sL}^{\cdot n+1})=-\lim_{k\to\infty}\frac{\dim\left(
\frac{H^0(X\times\A^1, \sO_{X\times\A^1}(-krK_{X\times\A^1/\A^1}))}{
H^0(X\times\A^1, \sO_{X\times\A^1}(-krK_{X\times\A^1/\A^1})\cdot\sI^k)}
\right)}{k^{n+1}/(n+1)!}.
\]
\end{proposition}

\begin{proof}
Let $\nu\colon\sX^\nu\to \sX$ be the normalization. Then 
$\alpha\colon(\sX^\nu, \nu^*\sL)\to\A^1$ is a normal semi test configuration of 
$(X, -rK_X)$. Set 
\[
\sY:=\Proj\bigoplus_{m\geq 0}\alpha_*(\nu^*\sL^{\otimes m})
\]
and let $\phi\colon\sX^\nu\to\sY$ be the natural morphism. Then there exist 
a positive integer $m$ and a line bundle $\sM$ on $\sY$ with a $\G_m$-action 
such that 
$\phi^*\sM$ is $\G_m$-equivariantly isomorphic to $\nu^*\sL^{\otimes m}$ and 
$(\sY, \sM)/\A^1$ is a normal test configuration of $(X, -mrK_X)$. 
Since $X$ is Ding semistable, we have $\Ding(\sY, \sM)\geq 0$. On the other hand, 
by Lemma \ref{crepant_lem}, we have 
$\Ding(\sY, \sM)=\Ding(\sX^\nu, \nu^*\sL^{\otimes m})$. 
Thus we have $\Ding(\sX^\nu, \nu^*\sL)\geq 0$ since 
$\Ding(\sX^\nu, \nu^*\sL^{\otimes m})=\Ding(\sX^\nu, \nu^*\sL)$ holds. 
Note that 
\begin{eqnarray*}
\sO_{\bar{\sX}^\nu}\left(r(K_{\bar{\sX}^\nu/\pr^1}+D_{(\sX^\nu, \nu^*\sL)})\right)
\simeq\nu^*\bar{\sL}^{\otimes(-1)}\\
\simeq\nu^*\sO_{\bar{\sX}}\left(r(\Pi^*K_{X\times\pr^1/\pr^1}+(1/r)E)\right). 
\end{eqnarray*}
Hence, for $c\in \R$, the pair $(\sX^\nu, D_{(\sX^\nu, \nu^*\sL)}+c\sX_0^\nu)$ 
is sub log canonical if and only if the pair 
$(X\times\A^1, \sI^{\cdot (1/r)}\cdot(t)^{\cdot c})$ is sub log canonical. 
Thus we have the equality
\[
\lct(\sX^\nu, D_{(\sX^\nu, \nu^*\sL)}; \sX_0^\nu)
=\lct(X\times\A^1, \sI^{\cdot (1/r)}; (t)).
\]
This implies that the pair $(X\times\A^1, \sI^{\cdot (1/r)}\cdot (t)^{\cdot d})$ is 
sub log canonical. The remaining part is trivial (see \cite[\S 3]{odk} for example). 
\end{proof}

\section{Ding semistability and filtered linear series}\label{ding_subscheme_section}

\subsection{The saturations of filtered linear series}\label{saturation_section}

We recall the definitions in \cite[\S 1]{BC}. 

\begin{definition}[{see \cite[\S 1]{BC}}]\label{filter_dfn}
Let $X$ be a projective variety, $L$ be a big line bundle on $X$, $V_\bullet$ be the 
complete graded linear series of $L$, that is, $V_r:=H^0(X, L^{\otimes r})$ for any 
$r\in \Z_{\geq 0}$. Let $\sF$ be a decreasing, left-continuous $\R$-filtration of the 
graded $\C$-algebra $V_\bullet$. 
\begin{enumerate}
\renewcommand{\theenumi}{\arabic{enumi}}
\renewcommand{\labelenumi}{(\theenumi)}
\item\label{filter_dfn1}
$\sF$ is said to be \emph{multiplicative} if 
\[
\sF^xV_r\otimes_\C\sF^{x'}V_{r'}\to \sF^{x+x'}V_{r+r'}
\]
holds for any $r$, $r'\in\Z_{\geq 0}$ and $x$, $x'\in\R$. 
\item\label{filter_dfn2}
$\sF$ is said to be \emph{linearly bounded} if $e_{\min}(V_\bullet, \sF)$, 
$e_{\max}(V_\bullet, \sF)\in\R$, where
\begin{eqnarray*}
e_{\min}(V_\bullet, \sF) & := & \liminf_{r\to\infty}
\left(\frac{\inf\{x\in\R\,|\,\sF^xV_r\neq V_r\}}{r}\right),\\
e_{\max}(V_\bullet, \sF) & := & \limsup_{r\to\infty}
\left(\frac{\sup\{x\in\R\,|\,\sF^xV_r\neq 0\}}{r}\right).
\end{eqnarray*}
\item\label{filter_dfn3}
Assume that $\sF$ is multiplicative. For any $x\in\R$, we set
\[
\vol(\sF V_\bullet^x):=\limsup_{r\to \infty}\frac{\dim\sF^{rx}V_r}{r^n/n!}, 
\]
where $n:=\dim X$. 
\end{enumerate}
\end{definition}

\begin{definition}\label{sat_dfn}
Let $X$ be a projective variety, $L$ be an ample line bundle on $X$, 
$V_\bullet$ be the complete graded linear series of $L$ and 
$\sF$ be a decreasing, left-continuous, 
multiplicative and linearly bounded $\R$-filtration of $V_\bullet$. 
For any $r\in \Z_{\geq 0}$ and $x\in\R$, we set 
\[
I_{(r,x)}^{\sF}:=I_{(r,x)}:=\Image(\sF^xV_r\otimes_\C L^{\otimes(-r)}\to\sO_X), 
\]
where the homomorphism is the evaluation homomorphism. Moreover, we set
$\bar{\sF}^xV_r:=H^0(X, L^{\otimes r}\cdot I_{(r,x)})$. 
\end{definition}

\begin{proposition}\label{saturation_prop}
Let $X$, $L$, $V_\bullet$ and $\sF$ be as in Definition \ref{sat_dfn}. 
\begin{enumerate}
\renewcommand{\theenumi}{\arabic{enumi}}
\renewcommand{\labelenumi}{(\theenumi)}
\item\label{saturation_prop1}
For any $r$, $r'\in\Z_{\geq 0}$ and $x$, $x'\in\R$, we have 
$I_{(r, x)}\cdot I_{(r', x')}\subset I_{(r+r', x+x')}$. 
\item\label{saturation_prop2}
For any $r\in\Z_{\geq 0}$ and $x\leq x'$, we have $I_{(r, x')}\subset I_{(r, x)}$. 
\item\label{saturation_prop3}
For any $r\in\Z_{\geq 0}$ and $x>r\cdot e_{\max}(V_\bullet, \sF)$, we have 
$\sF^x V_r=0$. In particular, $I_{(r, x)}=0$ holds. 
\item\label{saturation_prop4}
For any $e_-<e_{\min}(V_\bullet, \sF)$, there exists $r_1\in\Z_{>0}$ such that 
$\sF^{re_-}V_r=V_r$ and $I_{(r, re_-)}=\sO_X$ hold for any $r\geq r_1$. 
\item\label{saturation_prop5}
For any $r\in\Z_{\geq 0}$ and $x\in\R$, $\sF^x V_r\subset\bar{\sF}^x V_r$ holds. 
Moreover, the homomorphism $\bar{\sF}^x V_r\otimes_\C\sO_X\to L^{\otimes r}
\cdot I_{(r, x)}$ is surjective. 
\item\label{saturation_prop6}
$\bar{\sF}$ is also a decreasing, left-continuous, 
multiplicative and linearly bounded $\R$-filtration of $V_\bullet$. 
Moreover, we have 
\[
e_{\min}(V_\bullet, \sF)\leq e_{\min}(V_\bullet, \bar{\sF})
\leq e_{\max}(V_\bullet, \bar{\sF})=e_{\max}(V_\bullet, \sF).
\]
Furthermore, for any $r\in\Z_{\geq 0}$ and $x\in\R$, we have 
$I_{(r, x)}^{\sF}=I_{(r, x)}^{\bar{\sF}}$.
\end{enumerate}
\end{proposition}

\begin{proof}
\eqref{saturation_prop1}
Follows from the diagram
\[\xymatrix{
(\sF^x V_r\otimes_\C\sF^{x'}V_{r'})\otimes_\C L^{\otimes(-(r+r'))} 
\ar@{->>}[r] \ar[d] & I_{(r, x)}\cdot I_{(r', x')} \ar@{^{(}->}[r] & \sO_X \ar@{=}[d] \\
\sF^{x+x'}V_{r+r'}\otimes_\C L^{\otimes(-(r+r'))} \ar@{->>}[r]  & I_{(r+r', x+x')} 
\ar@{^{(}->}[r]& \sO_X.
}\]

\eqref{saturation_prop2}
This is obvious since $\sF$ is decreasing. 

\eqref{saturation_prop3}
(See \cite[Lemma 1.4]{BC}.) By the definition of $e_{\max}(V_\bullet, \sF)$, 
there exists $k\in\Z_{>0}$ such that $\sF^{kx}V_{kr}=0$. Thus $\sF^xV_r=0$ since 
$V_\bullet$ is an integral domain. 

\eqref{saturation_prop4}
By \cite[Example 1.2.22]{L1}, there exists $r_0\in \Z_{>0}$ such that 
the homomorphisms
\begin{eqnarray*}
V_r\otimes_\C V_{r'}\to V_{r+r'},\\
V_r\otimes_\C L^{\otimes (-r)}\to \sO_X
\end{eqnarray*}
are surjective for all $r$, $r'\geq r_0$. 
By the choice of $e_-$, there exist distinct prime numbers $p_1$, $p_2$ with 
$p_1$, $p_2\geq r_0$ such that $\sF^{p_i e_-}V_{p_i}=V_{p_i}$ for $i=1$, $2$. Set 
$r_1:=p_1p_2$. For any $r\geq r_1$, there exist $k_1$, $k_2\in\Z_{\geq 0}$ such that 
$r=k_1p_1+k_2p_2$ holds. Then $\sF^{re_-}V_r$ contains the image of 
$(\sF^{p_1 e_-}V_{p_1})^{\otimes k_1}\otimes_\C(\sF^{p_2 e_-}V_{p_2})^{\otimes k_2}
=V_{p_1}^{\otimes k_1}\otimes_\C V_{p_2}^{\otimes k_2}$. 
Thus $\sF^{re_-}V_r=V_r$ and $I_{(r, re_-)}=\sO_X$ hold. 

\eqref{saturation_prop5}
Consider the diagram
\[\xymatrix{
\sF^x V_r\otimes_\C \sO_X 
\ar@{->>}[r] \ar@{^{(}->}[d] & L^{\otimes r}\cdot I_{(r, x)}\ar@{^{(}->}[d] \\
H^0(X, L^{\otimes r})\otimes_\C\sO_X \ar[r]  & L^{\otimes r}.
}\]
By taking $H^0$, we get $\sF^x V_r\subset\bar{\sF}^x V_r$. From the diagram 
\[\xymatrix{
\sF^x V_r\otimes_\C \sO_X 
\ar@{->>}[r] \ar@{^{(}->}[d] & L^{\otimes r}\cdot I_{(r, x)}\ar@{=}[d] \\
\bar{\sF}^x V_r\otimes_\C \sO_X \ar[r]  & L^{\otimes r}\cdot I_{(r, x)},
}\]
we get the assertion. 

\eqref{saturation_prop6}
From \eqref{saturation_prop2}, $\bar{\sF}$ is decreasing, 
and obviously left-continuous. From \eqref{saturation_prop1}, $\bar{\sF}$ is 
multiplicative. From \eqref{saturation_prop5}, $e_{\min}(V_\bullet, \sF)\leq
e_{\min}(V_\bullet, \bar{\sF})$ and $e_{\max}(V_\bullet, \sF)\leq 
e_{\max}(V_\bullet, \bar{\sF})$ hold. Moreover, from \eqref{saturation_prop3}, 
$e_{\max}(V_\bullet, \sF)\geq e_{\max}(V_\bullet, \bar{\sF})$ holds. 
Thus $\bar{\sF}$ is linearly bounded. Moreover, the condition 
$I_{(r, x)}^{\sF}=I_{(r, x)}^{\bar{\sF}}$ follows from \eqref{saturation_prop5}. 
\end{proof}

\begin{definition}\label{saturation_dfn}
Let $X$, $L$, $V_\bullet$, $\sF$ be as in Definition \ref{sat_dfn}. 
\begin{enumerate}
\renewcommand{\theenumi}{\arabic{enumi}}
\renewcommand{\labelenumi}{(\theenumi)}
\item\label{saturation_dfn1}
The filtration $\bar{\sF}$ of $V_\bullet$ in Definition \ref{sat_dfn} is called 
the \emph{saturation} of $\sF$. 
\item\label{saturation_dfn2}
If $\sF^x V_r=\bar{\sF}^x V_r$ for any $r\in\Z_{\geq 0}$ and $x\in\R$, then we say that 
the filtration $\sF$ is \emph{saturated}. Note that, by Proposition \ref{saturation_prop}, 
for any $\sF$ in Definition \ref{sat_dfn}, the saturation $\bar{\sF}$ is saturated. 
\end{enumerate}
\end{definition}

\subsection{Test configurations from filtered linear series}\label{tc_filt_section}
In this section, we fix 
\begin{itemize}
\item
an $n$-dimensional $\Q$-Fano variety $X$ which is Ding semistable, 
\item
$r_0\in\Z_{>0}$ such that $-r_0K_X$ is Cartier, 
\item
$L:=\sO_X(-r_0K_X)$, 
\item
the complete graded linear series $V_\bullet$ of $L$, 
\item
a decreasing, left-continuous, multiplicative, linearly bounded $\R$-filtration 
$\sF$ of $V_\bullet$, and 
\item
$e_+$, $e_-\in\Z$ with $e_+>e_{\max}(V_\bullet, \sF)$ and 
$e_-<e_{\min}(V_\bullet, \sF)$. 
\end{itemize}

Set $e:=e_+-e_-$. Fix $r_1\in\Z_{>0}$ as in Proposition \ref{saturation_prop} 
\eqref{saturation_prop4}. For any $r\geq r_1$, we set 
\[
\sI_r:=I_{(r, re_+)}+I_{(r, re_+-1)}t^1+\cdots+I_{(r, re_-+1)}t^{re-1}+(t^{re})
\subset\sO_{X\times\A^1_t}.
\]

By Proposition \ref{saturation_prop}, $\{\sI_r\}_{r\geq r_1}$ 
is a graded family of coherent ideal sheaves. 
For any $r\geq r_1$, $k\in \Z_{\geq 0}$ and $j\in[kre_-, kre_+]\cap\Z$, we set 
\[
J_{(k; r, j)}:=\sum_{\substack{j_1+\cdots+j_k=j, \\
j_1,\dots,j_k\in[re_-, re_+]\cap\Z}}
I_{(r, j_1)}\cdots I_{(r, j_k)}. 
\]
By construction, 
\[
\sI_r^k=J_{(k; r, kre_+)}+J_{(k; r, kre_+-1)}t^1+\cdots+J_{(k; r, kre_-+1)}t^{kre-1}+(t^{kre})
\]
holds. Moreover, by Proposition \ref{saturation_prop} \eqref{saturation_prop5}, 
$J_{(k; r, j)}$ is the image of the homomorphism
\[
W_{(k; r, j)}\otimes_\C L^{\otimes(-kr)}\to\sO_X, 
\]
where $W_{(k;r, j)}$ is defined by the image of the homomorphism 
\[
\bigoplus_{\substack{j_1+\cdots+j_k=j, \\
j_1,\dots,j_k\in[re_-, re_+]\cap\Z
}}\bar{\sF}^{j_1}V_r\otimes_\C\cdots\otimes_\C\bar{\sF}^{j_k}V_r\to\bar{\sF}^j V_{kr}. 
\]

\begin{lemma}\label{W_lem}
For any $r\geq r_1$, $k\in\Z_{\geq 0}$ and $j\in[kre_-, kre_+]\cap\Z$, we have the 
following: 
\begin{enumerate}
\renewcommand{\theenumi}{\arabic{enumi}}
\renewcommand{\labelenumi}{(\theenumi)}
\item\label{W_lem1}
$W_{(k; r, j)}\subset H^0(X, L^{\otimes kr}\cdot J_{(k; r, j)})\subset\bar{\sF}^j V_{kr}$ 
holds. 
\item\label{W_lem2}
The homomorphism 
\[
H^0(X, L^{\otimes kr}\cdot J_{(k; r, j)})\otimes_\C\sO_X\to L^{\otimes kr}\cdot 
J_{(k; r, j)}
\]
is surjective. 
\end{enumerate}
\end{lemma}

\begin{proof}
From the homomorphism 
\[
W_{(k; r, j)}\otimes_\C\sO_X\to L^{\otimes kr}\cdot J_{(k; r, j)}, 
\]
we get the inclusion $W_{(k; r, j)}\subset H^0(X, L^{\otimes kr}\cdot J_{(k; r, j)})$. 
Furthermore, from the diagram
\[\xymatrix{
W_{(k; r, j)}\otimes_\C \sO_X 
\ar@{->>}[r] \ar@{^{(}->}[d] & L^{\otimes kr}\cdot J_{(k;r,j)}\ar@{=}[d] \\
H^0(X, L^{\otimes kr}\cdot J_{(k; r, j)})\otimes_\C \sO_X \ar[r] & 
L^{\otimes kr}\cdot J_{(k;r,j)},
}\]
we have proved \eqref{W_lem2}. Moreover, from the diagram
\[\xymatrix{
W_{(k; r, j)}\otimes_\C L^{\otimes(-kr)} 
\ar@{->>}[r] \ar@{^{(}->}[d] & J_{(k;r,j)} \ar@{^{(}->}[r] & \sO_X \ar@{=}[d] \\
\bar{\sF}^j V_{kr}\otimes_\C L^{\otimes(-kr)} \ar@{->>}[r]  & I_{(kr, j)} 
\ar@{^{(}->}[r]& \sO_X,
}\]
we have $J_{(k; r, j)}\subset I_{(kr, j)}$. Thus we have proved \eqref{W_lem1}. 
\end{proof}

For any $r\geq r_1$, let 
\begin{itemize}
\item
$\Pi_r\colon\sX_r\to X\times\A^1$ be the blowup along $\sI_r$, 
\item
$E_r\subset\sX_r$ be the Cartier divisor defined by 
$\sO_{\sX_r}(-E_r)=\sI_r\cdot\sO_{\sX_r}$, and
\item
$\sL_r:=\Pi_r^*\sO_{X\times\A^1}(-rr_0K_{X\times\A^1/\A^1})\otimes
\sO_{\sX_r}(-E_r)$.
\end{itemize}

\begin{lemma}\label{stc_lemma}
$\sL_r$ is semiample over $\A^1$. 
Thus $(\sX_r, \sL_r)/\A^1$ is a 
semi test configuration of $(X, -rr_0K_X)$. 
\end{lemma}

\begin{proof}
(See also \cite[Lemma 3.4]{fjt1}.)
Let $\alpha\colon\sX_r\to \A^1$ and $p_2\colon X\times\A^1\to \A^1$ be 
the natural morphisms. 
For any $k\in\Z_{\geq 0}$, by Lemma \ref{W_lem} \eqref{W_lem2}, we have
\begin{eqnarray*}
&&H^0(X\times\A^1, \sO_{X\times\A^1}(-krr_0K_{X\times\A^1/\A^1})\cdot\sI_r^k)
\otimes_{\C[t]}\sO_{X\times\A^1}\\
&=&\Biggl(\sum_{j=0}^{kre-1}t^j\cdot 
H^0(X, L^{\otimes kr}\cdot J_{(k; r, kre_+-j)})\\
&&+\sum_{j\geq kre}t^j\cdot H^0(X, L^{\otimes kr})\Biggr)
\otimes_{\C[t]}\sO_{X\times\A^1}\\
&\twoheadrightarrow&\sum_{j=0}^{kre-1}t^j\cdot L^{\otimes kr}\cdot 
J_{(k; r, kre_+-j)}+\sum_{j\geq kre}t^j\cdot L^{\otimes kr}\\
&=&\sO_{X\times\A^1}(-krr_0 K_{X\times\A^1/\A^1})\cdot\sI_r^k.
\end{eqnarray*}
Therefore, by \cite[Lemma 5.4.24]{L1}, for any $k\gg 0$, we have
\begin{eqnarray*}
&&\alpha^*\alpha_*\sL_r^{\otimes k}\\
&=&\Pi_r^*(p_2)^*(p_2)_*\left(
\sO_{X\times\A^1}(-krr_0 K_{X\times\A^1/\A^1})\cdot\sI_r^k\right)\\
&=&\Pi_r^*\left(H^0(X\times\A^1, 
\sO_{X\times\A^1}(-krr_0 K_{X\times\A^1/\A^1})\cdot\sI_r^k)\otimes_{\C[t]}
\sO_{X\times\A^1}\right)\\
&\twoheadrightarrow& \Pi_r^*\left(
\sO_{X\times\A^1}(-krr_0 K_{X\times\A^1/\A^1})\cdot\sI_r^k\right)\\
&\twoheadrightarrow& \Pi_r^*\sO_{X\times\A^1}(-krr_0K_{X\times\A^1/\A^1})
\otimes\sO_{\sX_r}(-kE_r)=\sL_r^{\otimes k}.
\end{eqnarray*}
Thus $\sL_r$ is semiample over $\A^1$. 
\end{proof}

Thus, by Proposition \ref{ding_flag_prop}, the pair 
$(X\times\A^1, \sI_r^{\cdot (1/(rr_0))}\cdot (t)^{\cdot d_r})$ is 
sub log canonical, where 
\[
d_r:=1+\frac{(\bar{\sL}_r^{\cdot n+1})}{(n+1)r^{n+1}r_0^{n+1}((-K_X)^{\cdot n})}.
\]
Set
\[
w_r(k):=-\dim\left(\frac{
H^0(X\times\A^1, \sO_{X\times\A^1}(-krr_0K_{X\times\A^1/\A^1}))}{
H^0(X\times\A^1, \sO_{X\times\A^1}(-krr_0K_{X\times\A^1/\A^1})\cdot\sI_r^k)}
\right).
\]
Then 
\[
(\bar{\sL}_r^{\cdot n+1})=\lim_{k\to\infty}\frac{w_r(k)}{k^{n+1}/(n+1)!}
\]
holds by Proposition \ref{ding_flag_prop}.
We set 
\begin{eqnarray*}
v_r(k)&:=&\sum_{j=kre_-+1}^{kre_+}h^0(X, L^{\otimes kr}
\cdot J_{(k; r, j)}), \\
A_r&:=&\lim_{k\to\infty}\frac{v_r(k)}
{k^{n+1}r^{n+1}r_0^{n+1}/n!}.
\end{eqnarray*}
Since $w_r(k)=-kre\cdot 
h^0(X, L^{\otimes kr})+v_r(k)$, 
the limit in the definition of $A_r$ actually exists. Note that 
$d_r=1-e/r_0+A_r/((-K_X)^{\cdot n})$. 

\begin{lemma}[{cf.\ \cite[Theorem 1.14]{BC}}]\label{limit_lem}
We have
\[
\lim_{r\to\infty}A_r=\frac{1}{r_0^{n+1}}\int_{e_-}^{e_+}\vol(\bar{\sF}V_\bullet^x)dx.
\]
\end{lemma}

\begin{proof}
Take any $r\geq r_1$. For $k\in\Z_{\geq 0}$, set
\[
W_{r, k}:=\Image(V_r^{\otimes k}\to V_{kr})=V_{kr}.
\]
Moreover, we consider the $\R$-filtration $\sG$ of the complete graded linear series 
$W_{r, \bullet}$ of $L^{\otimes rr_0}$, 
where $\sG^x W_{r, k}$ is defined by the image of the homomorphism
\[
\sum_{\substack{x_1+\cdots+x_k=x, \\
x_1,\dots,x_k\in\R}}\bar{\sF}^{x_1}V_r\otimes_\C\cdots\otimes_\C\bar{\sF}^{x_k}V_r
\to\bar{\sF}^x V_{kr}.
\]

\begin{claim}\label{G_claim}
\begin{enumerate}
\renewcommand{\theenumi}{\arabic{enumi}}
\renewcommand{\labelenumi}{(\theenumi)}
\item\label{G_claim1}
$\sG$ is a decreasing, left-continuous, multiplicative, linearly bounded 
$\R$-filtration of $W_{r, \bullet}$.
\item\label{G_claim2}
We have
\[
re_-\leq e_{\min}(W_{r, \bullet}, \sG)
\leq e_{\max}(W_{r, \bullet}, \sG)\leq re_+.
\]
\item\label{G_claim3}
For any $k\in\Z_{\geq 0}$ and $j\in[kre_-, kre_+]\cap\Z$, we have
\[
\sG^{j+k-1} W_{r, k}\subset H^0(X, L^{\otimes kr}\cdot J_{(k; r, j)})\subset
\bar{\sF}^j V_{kr}.
\]
\end{enumerate}
\end{claim}

\begin{proof}[Proof of Claim \ref{G_claim}]
\eqref{G_claim1}
We check that $\sG$ is left-continuous. For any $i\in[1, \dim V_r]\cap\Z$ and 
$k\in\Z_{\geq 0}$, we set
\begin{eqnarray*}
e_{r, i} & := & \sup\{x\in\R\,|\, \dim\bar{\sF}^x V_r\geq i\}, \\
\E_{r,k} & := & \left\{\sum_{i=1}^k e_{r, j_i}\, \bigg| \, 
j_1,\dots,j_k\in[1,\dim V_r]\cap\Z\right\}\subset\R.
\end{eqnarray*}
Moreover, we set $e_{r, 0}:=+\infty$ and $e_{r, \dim V_r +1}:=-\infty$ for convenience. 
Take any $x\in\R$. Then 
$x\not\in\{e'+\varepsilon'\,|\,e'\in\E_{r,k}$, $\varepsilon'\in(0, \varepsilon]\}$ holds
for any $0<\varepsilon\ll 1$. Take such $\varepsilon$. It is enough to show 
$\sG^{x-\varepsilon}W_{r, k}\subset \sG^x W_{r, k}$ for proving that $\sG$ is 
left-continuous. Pick any $x'_1,\dots, x'_k\in\R$ with $x'_1+\cdots+x'_k=x-\varepsilon$. 
For any $1\leq i\leq k$, there exists a unique $0\leq j_i\leq \dim V_r$ such that 
$x'_i\in(e_{r, j_i+1}, e_{r, j_i}]$. By the choice of $\varepsilon$, we have 
$\sum_{i=1}^k(e_{r, j_i}-x'_i)\geq\varepsilon$. Thus there exist $x_1,\dots, x_k\in\R$ 
such that $x_1+\cdots+x_k=x$ and $x_i\in(e_{r, j_i+1}, e_{r, j_i}]$ for any $1\leq i\leq k$. 
Since $\bar{\sF}^{x'_1}V_r\otimes_\C\cdots\otimes_\C\bar{\sF}^{x'_k}V_r=
\bar{\sF}^{x_1}V_r\otimes_\C\cdots\otimes_\C\bar{\sF}^{x_k}V_r$, we get 
$\sG^{x-\varepsilon}W_{r, k}\subset \sG^x W_{r, k}$. The remaining assertions are trivial. 

\eqref{G_claim2}
Pick any $k\in\Z_{\geq 0}$. For any $x< kre_-$, we have 
$\bar{\sF}^{x/k}V_r=V_r$. Thus $\sG^x W_{r, k}=W_{r, k}$ and this implies that 
$re_-\leq e_{\min}(W_{r, \bullet}, \sG)$. 
For any $x> kre_+$ and for any $x_1,\dots,x_k\in\R$ 
with $x_1+\cdots+x_k=x$, there exists $1\leq i\leq k$ such that 
$x_i>re_+$. Thus 
$\bar{\sF}^{x_1}V_r\otimes_\C\cdots\otimes_\C\bar{\sF}^{x_k}V_r=0$ and 
this implies that 
$e_{\max}(W_{r, \bullet}, \sG)\leq re_+$.

\eqref{G_claim3}
By Lemma \ref{W_lem} \eqref{W_lem1}, it is enough to show that $\sG^{j+k-1}W_{r, k}
\subset W_{(k; r, j)}$. Take any $x_1,\dots, x_k\in\R$ with $x_1+\cdots+x_k=j+k-1$. 
Then $\lfloor x_1\rfloor+\cdots+\lfloor x_k\rfloor\geq j$.
Thus the image of $\bar{\sF}^{x_1}V_r\otimes_\C\cdots\otimes_\C\bar{\sF}^{x_k}V_r$ is 
contained in $W_{(k; r, j)}$.
\end{proof}

By Claim \ref{G_claim} \eqref{G_claim3}, we get 
\begin{eqnarray*}
\int_{e_-+1/r}^{e_++1/r}\frac{\dim\sG^{krx}W_{r, k}}{k^nr^n/n!}dx
&\leq&
\sum_{j=kre_-+1}^{kre_+}
\frac{h^0(X, L^{\otimes kr}\cdot J_{(k; r, j)})}{k^{n+1}r^{n+1}/n!}\\
&\leq&
\int_{e_-}^{e_+}\frac{\dim\bar{\sF}^{krx}V_{kr}}{k^nr^n/n!}dx. 
\end{eqnarray*}
We note that both $\dim\bar{\sF}^{krx}V_{kr}$ and $\dim\sG^{krx+k}W_{r,k}$ are 
Lebesgue measurable on $x\in[e_-, e_+]$ since both are monotone decreasing functions. 
For any $x\in[e_-, e_+]\setminus\{e_{\max}(V_\bullet, \bar{\sF})\}$, the limit
\[
\lim_{k\to\infty}\frac{\dim\bar{\sF}^{kx}V_k}{k^n/n!}
\]
exists by \cite[Lemma 1.6]{BC}, \cite[Theorem 2.13]{LM} and Proposition 
\ref{saturation_prop} \eqref{saturation_prop3}. Hence, for any $r\geq r_1$, we have
\[
\lim_{k\to\infty}\frac{\dim\bar{\sF}^{krx}V_{kr}}{k^nr^n/n!}=\vol(\bar{\sF}V_\bullet^x).
\]
From dominated convergence, we have
\[
\lim_{k\to\infty}\int_{e_-}^{e_+}
\frac{\dim\bar{\sF}^{krx}V_{kr}}{k^nr^n/n!}dx=
\int_{e_-}^{e_+}\vol(\bar{\sF}V_\bullet^x)dx.
\]
By the same argument, the limit 
\[
\lim_{k\to\infty}\frac{\dim\sG^{krx}W_{r, k}}{k^nr^n/n!}
\]
exists for any $x\in[e_-, e_+]\setminus\{r^{-1}\cdot e_{\max}(W_{r, \bullet}, \sG)\}$ and 
\[
\lim_{k\to\infty}\int_{e_-+1/r}^{e_++1/r}
\frac{\dim\sG^{krx}W_{r, k}}{k^nr^n/n!}dx=
\int_{e_-}^{e_+}\frac{\vol(\sG W_{r, \bullet}^{rx})}{r^n}dx-\frac{r_0^n((-K_X)^{\cdot n})}{r}
\]
holds since we have
\[
\int_{e_-+1/r}^{e_++1/r}\frac{\dim\sG^{krx}W_{r, k}}{k^nr^n/n!}dx=
\int_{e_-}^{e_+}\frac{\dim\sG^{krx}W_{r, k}}{k^nr^n/n!}dx
-\frac{h^0(X, L^{\otimes kr})}{k^nr^{n+1}/n!}.
\]
Thus we get
\[
\int_{e_-}^{e_+}\frac{\vol(\sG W_{r, \bullet}^{rx})}{r^n}dx-
\frac{r_0^n((-K_X)^{\cdot n})}{r} \leq r_0^{n+1}A_r \leq
\int_{e_-}^{e_+}\vol(\bar{\sF}V_\bullet^x)dx.
\]

By \cite[Lemma 1.6]{BC} and \cite[Theorem 3.5]{LM}, for any $x\in[e_-, e_+]\setminus
\{e_{\max}(V_\bullet, \bar{\sF})\}$, we have
\[
\lim_{r\to\infty}\frac{\vol(\sG W_{r, \bullet}^{rx})}{r^n}=\vol(\bar{\sF}V_\bullet^x).
\]
Again by dominated convergence, we have
\[
\lim_{r\to\infty}\int_{e_-}^{e_+}\frac{\vol(\sG W_{r, \bullet}^{rx})}{r^n}dx
=\int_{e_-}^{e_+}\vol(\bar{\sF}V_\bullet^x)dx. 
\]
Therefore the limit $\lim_{r\to\infty}A_r$ exists and is equal to the right-hand side 
of Lemma \ref{limit_lem}. 
\end{proof}

By Lemmas \ref{gr_lc_lem} \eqref{gr_lc_lem1} and \ref{limit_lem}, the pair 
$(X\times\A^1, \sI_\bullet^{\cdot (1/r_0)}\cdot(t)^{\cdot d_\infty})$ is 
sub log canonical, where
\[
d_\infty:=1-\frac{e}{r_0}+\frac{1}{r_0^{n+1}((-K_X)^{\cdot n})}\int_{e_-}^{e_+}
\vol(\bar{\sF}V_\bullet^x)dx.
\]
Consequently, we have proved the following: 

\begin{thm}\label{filt_thm}
Let $X$, $r_0$, $L$, $V_\bullet$, $\sF$, $e_+$, $e_-$ be as in the beginning of Section 
\ref{tc_filt_section}. Then the pair $(X\times\A^1, 
\sI_\bullet^{\cdot (1/r_0)}\cdot(t)^{\cdot d_\infty})$ is 
sub log canonical, where
\begin{eqnarray*}
\sI_r&=&I_{(r, re_+)}^\sF+I_{(r, re_+-1)}^\sF t^1+\cdots+
I_{(r, re_-+1)}^\sF t^{r(e_+-e_-)-1}+(t^{r(e_+-e_-)}), \\
d_\infty&=&1-\frac{e_+-e_-}{r_0}+\frac{1}{r_0^{n+1}((-K_X)^{\cdot n})}\int_{e_-}^{e_+}
\vol(\bar{\sF}V_\bullet^x)dx.
\end{eqnarray*}
\end{thm}

\subsection{Ding semistability along subschemes}\label{ding_sub_section}

\begin{thm}\label{ding_thm}
Let $X$ be an $n$-dimensional $\Q$-Fano variety. 
Assume that $X$ is Ding semistable. 
Take any nonempty proper closed subscheme $\emptyset\neq Z\subsetneq X$ 
corresponds to an ideal sheaf $0\neq I_Z\subsetneq\sO_X$. 
Let $\sigma\colon\hat{X}\to X$ be the blowup along $Z$, let $F\subset\hat{X}$ be 
the Cartier divisor defined by the equation $\sO_{\hat{X}}(-F)=I_Z\cdot\sO_{\hat{X}}$. 
Then we have $\beta(Z)\geq 0$, where 
\[
\beta(Z):=\lct(X;I_Z)\cdot((-K_X)^{\cdot n})-
\int_0^\infty\vol_{\hat{X}}\left(\sigma^*(-K_X)-xF\right)dx.
\]
\end{thm}

\begin{proof}
Fix $r_0\in\Z_{>0}$ with $-r_0K_X$ Cartier and set $L:=\sO_X(-r_0K_X)$. 
Let $V_\bullet$ be the complete graded linear series of $L$. Consider the 
$\R$-filtration $\sF$ of $V_\bullet$ defined by 
\[
\sF^x V_r:=\begin{cases}
H^0(X, L^{\otimes r}\cdot I_Z^{\lceil x\rceil}) & \text{if }x\in\R_{\geq 0},\\
V_r & \text{otherwise}.
\end{cases}
\]
Then $\sF$ is a decreasing, left-continuous, multiplicative and linearly bounded 
$\R$-filtration of $V_\bullet$. In fact, we can immediately check that 
$e_{\min}(V_\bullet, \sF)=0$ and $e_{\max}(V_\bullet, \sF)=r_0\tau_Z$. 
We note that the filtration $\sF$ is saturated. Indeed, the homomorphism
\[
\sF^x V_r\otimes_\C L^{\otimes (-r)}\twoheadrightarrow I_{(r, x)}
\]
induces the inclusion $I_{(r, x)}\subset I_Z^{\lceil x\rceil}$ for any $x\in\R_{\geq 0}$. 
Thus $\bar{\sF}^x V_r=H^0(X, L^{\otimes r}\cdot I_{(r, x)})\subset\sF^x V_r$. 

Fix $e_+$, $e_-\in\Z$ with $e_+>r_0\tau_Z$ and $e_-<0$. 
By Theorem \ref{filt_thm}, the pair 
$(X\times\A^1, \sI_\bullet^{\cdot (1/r_0)}\cdot(t)^{\cdot d_\infty})$ is 
sub log canonical, where 
\begin{eqnarray*}
\sI_r&=&I_{(r, re_+)}+I_{(r, re_+-1)}t^1+\cdots+
I_{(r, re_-+1)}t^{r(e_+-e_-)-1}+(t^{r(e_+-e_-)}), \\
d_\infty&=&1-\frac{e_+-e_-}{r_0}+\frac{1}{r_0^{n+1}((-K_X)^{\cdot n})}\int_{e_-}^{e_+}
\vol(\sF V_\bullet^x)dx.
\end{eqnarray*}
Note that 
\[
\vol(\sF V_\bullet^x)=\begin{cases}
r_0^n\vol_{\hat{X}}(\sigma^*(-K_X)-(x/r_0)F) & \text{if }x\in\R_{\geq 0}, \\
r_0^n((-K_X)^{\cdot n}) & \text{otherwise}.
\end{cases}
\]
Thus 
$d_\infty=1-\tau+S$
holds, where $\tau:=e_+/r_0$ and 
\[S:=\frac{1}{((-K_X)^{\cdot n})}\int_0^\infty\vol_{\hat{X}}(\sigma^*(-K_X)
-xF)dx.
\] 
Moreover, for any $r\gg 0$, 
\begin{eqnarray*}
\sI_r\subset I_Z^{re_+}+I_Z^{re_+-1}t^1+\cdots+I_Z^1t^{re_+-1}+(t^{re_+})
=(I_Z+(t))^{re_+}.
\end{eqnarray*}
By Lemma \ref{gr_lc_lem} \eqref{gr_lc_lem2}, the pair 
$(X\times\A^1, (I_Z+(t))^{\cdot \tau}\cdot (t)^{\cdot d_\infty})$ is sub log canonical. 

Let $\theta\colon\sY\to X\times\A^1$ be a common log resolution of $X\times\A^1$, 
$I_Z+(t)$ and $(t)$, that is, $\sY$ is smooth, $(I_Z+(t))\cdot\sO_\sY=:\sO_\sY(-F_1)$, 
$(t)\cdot\sO_\sY=:\sO_\sY(-F_2)$ satisfy that $\Exc(\theta)$, $\Exc(\theta)+F_1+F_2$ 
are divisors with simple normal crossing supports. 
For any $c_1$, $c_2\in\R$, we set 
\begin{eqnarray*}
&&\sJ\left(X\times\A^1, (I_Z+(t))^{\cdot c_1}\cdot(t)^{\cdot c_2}\right)\\
&:=&\theta_*\sO_\sY\left(\lceil K_\sY-\theta^*K_{X\times\A^1}
-c_1F_1-c_2F_2\rceil\right), 
\end{eqnarray*}
where $\lceil K_\sY-\theta^*K_{X\times\A^1}-c_1F_1-c_2F_2\rceil$ 
is the smallest $\Z$-divisor which contains 
$K_\sY-\theta^*K_{X\times\A^1}-c_1F_1-c_2F_2$.
If $c_1$, $c_2\in\R_{\geq 0}$, then this is nothing but the multiplier ideal sheaf of 
the pair $(X\times\A^1, (I_Z+(t))^{\cdot c_1}\cdot (t)^{\cdot c_2})$ 
(see \cite[\S 9]{L2} or \cite{T}).
Take any $0<\varepsilon\ll 1$. Then we have 
\[
\sO_{X\times\A^1}\subset
\sJ\left(X\times\A^1, (I_Z+(t))^{\cdot (1-\varepsilon)\tau}\cdot
(t)^{\cdot (1-\varepsilon)d_\infty}\right) 
\]
since $X\times\A^1$ is log terminal. 
Pick any positive integer $N$ with $(1-\varepsilon)d_\infty+N>0$. 
By the definition of 
$\sJ\left(X\times\A^1, (I_Z+(t))^{\cdot c_1}\cdot(t)^{\cdot c_2}\right)$, we have 
\[
(t^N)\subset
\sJ\left(X\times\A^1, (I_Z+(t))^{\cdot (1-\varepsilon)\tau}\cdot
(t)^{\cdot (1-\varepsilon)d_\infty+N}\right)\subset\sO_{X\times\A^1}. 
\]
By \cite[Theorem 3.2]{T} and \cite[Remark 9.5.23]{L2}, we have
\begin{eqnarray*}
&&\sJ\left(X\times\A^1, (I_Z+(t))^{\cdot (1-\varepsilon)\tau}\cdot
(t)^{\cdot (1-\varepsilon)d_\infty+N}\right)\\
&=&\sum_{0\leq\tau'\leq(1-\varepsilon)\tau}
\sJ\left(X\times\A^1, I_Z^{\cdot \tau'}\cdot
(t)^{\cdot (1-\varepsilon)(d_\infty+\tau)-\tau'+N}\right)\\
&=&\sum_{0\leq\tau'\leq(1-\varepsilon)\tau}
\sJ(X, I_Z^{\cdot \tau'})\cdot
\left(t^{\lfloor(1-\varepsilon)(d_\infty+\tau)-\tau'\rfloor+N}\right),
\end{eqnarray*}
where $\sJ(X, I_Z^{\cdot \tau'})$ is the multiplier ideal sheaf of the pair 
$(X, I_Z^{\cdot \tau'})$.
This implies that 
\[
\sO_X=\sum_{\tau'>S-\varepsilon(1+S)}\sJ(X, I_Z^{\cdot \tau'})
\]
since $(1-\varepsilon)(d_\infty+\tau)-1=S-\varepsilon(1+S)$. 
Therefore we get the inequality $\lct(X; I_Z)\geq S$. 
\end{proof}

\begin{remark}\label{eta_beta_rmk}
Assume that $X$ is smooth. If 
$Z$ is a reduced divisor with $(X, Z)$ log canonical 
(resp.\ $Z$ is a smooth subvariety with \cite[Assumption 3.1]{fjt2}), 
then the value $\beta(Z)$ 
is equal to the value 
$\eta(Z)$ in \cite[Definition 1.1]{fjt1} (resp.\ in \cite[Remark 3.10]{fjt2}). 
\end{remark}

\section{Proofs}\label{proof_section}

\begin{thm}\label{Pn_thm}
Let $X$ be an $n$-dimensional $\Q$-Fano variety which is Ding semistable. 
Then we have $((-K_X)^{\cdot n})\leq (n+1)^n$. Moreover, if we further assume that 
$X$ is smooth and $((-K_X)^{\cdot n})=(n+1)^n$, then $X$ is isomorphic to the 
projective space $\pr^n$. 
\end{thm}

\begin{proof}
We can assume that $n\geq 2$. 
Take any smooth closed point $p\in X$. Let $\sigma\colon\hat{X}\to X$ be the 
blowup along $p$ and let $F$ be the exceptional divisor of $\sigma$. By 
Theorem \ref{ding_thm}, we have 
\[
n\cdot((-K_X)^{\cdot n})\geq \int_0^\infty\vol_{\hat{X}}(\sigma^*(-K_X)-xF)dx.
\]
On the other hand, by Theorem \ref{sesh_thm} \eqref{sesh_thm1}, we have 
\begin{eqnarray*}
\int_0^\infty\vol_{\hat{X}}(\sigma^*(-K_X)-xF)dx&\geq&
\int_0^{\sqrt[n]{((-K_X)^{\cdot n})}}(((-K_X)^{\cdot n})-x^n)dx\\
&=&\sqrt[n]{((-K_X)^{\cdot n})}\cdot\frac{n}{n+1}((-K_X)^{\cdot n}). 
\end{eqnarray*}
Hence we get the inequality $(n+1)^n\geq ((-K_X)^{\cdot n})$. 
Assume that $(n+1)^n=((-K_X)^{\cdot n})$. 
Then 
\[
\vol_{\hat{X}}(\sigma^*(-K_X)-xF)=(n+1)^n-x^n
\]
for all $x\in[0, n+1]$. Thus, by Theorem \ref{sesh_thm} \eqref{sesh_thm2}, 
we have $\varepsilon_p=n+1$. If $X$ is smooth, this implies that $X\simeq\pr^n$ 
by \cite{CMSB} and \cite{kebekus} (see also \cite{BS}). 
\end{proof}

\begin{proof}[Proof of Theorem \ref{mainthm}]
This is an immediate consequence of Theorems \ref{berman_thm} and \ref{Pn_thm}. 
\end{proof}

\end{document}